\documentclass[draft,12pt,reqno,a4paper]{amsart}

\usepackage[margin=3cm,footskip=1cm]{geometry}

\usepackage{amssymb} \usepackage{amsmath} \usepackage{amsfonts}
\usepackage{amsthm} \usepackage{mathtools}

\usepackage{enumerate} \usepackage[latin1]{inputenc}
\usepackage[shortlabels]{enumitem}

\usepackage{color}
\usepackage{graphicx} \usepackage{verbatim}

\newcommand{\Ran}{{\operatorname{Ran}}}

\newcommand{\cs}{{\rm the Cauchy-Schwarz inequality }}
\newcommand{\cS}{{\rm the Cauchy-Schwarz inequality}}
\newcommand{\ad}{{\operatorname{ad}}}

\newcommand{\Opw}{{\operatorname{Op\!^w\!}}}
\newcommand{\N}{{\mathbb{N}}} 
\newcommand{\R}{{\mathbb{R}}} 
\newcommand{\C}{{\mathbb{C}}}

 \renewcommand{\c}{{\rm c}}
\newcommand{\e}{{\rm e}} 
\newcommand{\adop}{{\rm ad}}
\newcommand{\lop}{{\rm l}}
\newcommand{\rop}{{\rm r}}
  \renewcommand{\i}{{\rm i}}
\renewcommand{\d}{{\rm d}} 

 \newcommand{\righ}{{\rm right}}
\newcommand{\lef}{{\rm left}} \newcommand{\unif}{{\rm unif}}

\renewcommand{\Re}{{\rm Re}\,} \renewcommand{\Im}{{\rm Im}\,}

\DeclarePairedDelimiter\inp\langle\rangle

\newcommand\parb[2][]{#1 \big ( #2#1\big )} \newcommand\parbb[2][]{#1
  \Big ( #2#1\Big )} 
 \newcommand{\grad}{{\rm grad }\,}
 
 \renewcommand{\exp}{{\rm exp}}
\newcommand{\mand}{\text{ and }} 
\newcommand{\mforall}{\text{ for all }}

\theoremstyle{plain}
\newtheorem{thm}{Theorem}[section]

\newtheorem{lemma}[thm]{Lemma} 
\theoremstyle{definition}

\newtheorem{remarks}[thm]{Remarks}
 
\newtheorem*{remarks*}{Remarks}
\newtheorem*{remark*}{Remark}
\newtheorem*{defn*}{Definition}

\numberwithin{equation}{section}

\title{Decay of eigenfunctions of elliptic PDE's}

\author{I. Herbst}
\address[I.  Herbst]{Department of Mathematics \\
  University of Virginia \\
  Charlottesville \\
  VA 22904\\ U.S.A.}
\email{iwh@virginia.edu}

  \author{E. Skibsted} \address[E. Skibsted]{Institut for Matematiske
  Fag \\
  Aarhus Universitet\\ Ny Munkegade 8000 Aarhus C, Denmark}
\email{skibsted@imf.au.dk}

\begin{document}

  \begin{abstract} We study exponential  decay of eigenfunctions of self-adjoint 
    higher order elliptic operators on  $\R^d$.  We show that the possible critical
    decay rates are determined algebraically.   In addition we show absence of
    super-exponentially decaying eigenfunctions and a refined
    exponential 
    upper bound.

\end{abstract}


\maketitle

\tableofcontents

\section{Introduction and results}\label{subsec:result}
Consider  a real elliptic polynomial $Q$ of degree $q$ on  $\R^d$.
We consider  the operator
$H=Q(p)+V(x)$, $p=-\i \nabla$, on $L^2=L^2(\R^d)$ with  $V$  bounded
and  measurable  
and with $\lim_{|x| \to \infty} V(x) = 0$. 

We will mostly  assume there is a splitting of $V,$  $V=V_1+V_2$, into real-valued bounded functions, $V_1$
smooth and $V_2$  measurable, with additional assumptions depending on the result.

 For a given  
$\lambda\in \R$  the energy surface
  \begin{align*}
  S_{\lambda}=\{(x,\xi)\in \R^d\times \R^d|Q(\xi)=\lambda\}
\end{align*} is by definition regular if $\lambda$ is not a critical value of $Q$, that is if

\begin{align}\label{eq:14}
  \nabla Q(\xi)\neq 0 \text{ on } S_{\lambda}.
\end{align}
We will need this condition in one of our results.
 
Suppose
 $(H-\lambda)\phi=0$,
  $\phi\in L^2$.  The {\it critical  decay rate} is defined as 
  \begin{align*}
    \sigma_\c=\sup\{\sigma\geq 0|\e^{\sigma|x|}\phi\in L^2\}.
  \end{align*} 

In this paper we shall study this notion of decay rate for  eigenfunctions, cf. previous works for the Laplacian \cite{CT,FHH2O1,
FHH2O3, FHH2O2} corresponding to the case $Q(\xi)=\xi^2$. In particular we
give necessary (phase-space) conditions for a positive number  $\sigma$ to be the critical  decay
rate (generalizing a result for the Laplacian), and we shall give a
refined exponential upper bound (actually a new result for the Laplacian). This 
set of conditions defines  a notion of {\it exceptional points}
 playing the same role as a certain set for the
$N$-body problem related to the set of thresholds \cite{FH, IS1}.  
However these conditions are very
different from what could  be expected from this analogous problem. More precisely  the critical decay rate is not computable in terms  of the critical values of
$Q$ which are the exceptional energies for the Mourre estimate \cite{Mo} if we use the ``natural" conjugate operator $A = x\cdot \nabla Q(p) + \nabla Q(p)\cdot x $. On the other hand, as we shall see,  $\sigma_\c>0$ at
non-critical values. A similar result, although in 
different settings, appears in \cite{MP1, MP2}.  

To be a little more precise
about the analogy of our problem with the N-body problem let us add a
sufficiently decaying  $N$-particle interaction  potential to the N-body Hamiltonian (which produces a potential decaying in the whole space after the center of mass motion is removed).  Then the possible decay rates are independent of this potential just as they are independent of our decaying $V$ for the operator $H=Q(p) + V$.

To define the exceptional points we need the following notation: For $\omega \in \R^d$ we let 
  $P_{\perp}(\omega)=I-|\omega\rangle\langle \omega|$ (defined in terms
  of inner product brackets). 
\begin{defn*} Let $\lambda\in \R$ be given. The set of  exceptional points $\Sigma_{\rm
  exc}(\lambda)$ is the set of $ \sigma\in (0,\infty)$ for which 
  there exists
  $(\omega,\xi)\in S^{d-1}\times \R^d$ satisfying the equations
  \begin{subequations}
  \begin{align}\label{eq:10}
    Q(\xi+\i \sigma \omega)&=\lambda,\\
P_{\perp}(\omega)\nabla_\xi Q(\xi +\i \sigma \omega)&= 0.\label{eq:11} 
  \end{align}  
  \end{subequations}   
\end{defn*}

Another major subject of this paper is absence of
super-exponentially decaying  eigenfunctions (corresponding to the case
$\sigma_\c=\infty$), cf. \cite{FHH2O1, VW, IS2}. We show  absence of such
states  under somewhat strong decay conditions on the
potential.
\subsection{Results}\label{subsec:Results}
Suppose
  $(H-\lambda)\phi=0$,
  $\phi\in L^2$, with corresponding critical decay rate
  $\sigma_\c$. Let $\Ran Q = \{Q(\xi)| \xi \in \R^d\}$.  Our main results read:

\begin{thm}\label{thm:start}
Under either of the following two conditions we can conclude that $\sigma_\c >0$:
\begin{enumerate}[1)]
\item\label{item:1}
$\lambda \notin {\Ran }Q$  and $V(x) = o(1)$.
\item\label{item:2}
$\lambda  \in {\Ran }Q$ but $\lambda$ is not a critical value of $Q$ and in addition 
\begin{align*}
&\forall \alpha: \partial^{\alpha}V_1(x) = o(|x|^{- |\alpha|}),\\
&V_2(x) = o(|x|^{-1}).
\end{align*}
\end{enumerate}
\end{thm}

\begin{thm}\label{thm:12.7.2.7.9}  Suppose  $0<\sigma_\c< \infty$.
\begin{enumerate}[i)]
\item \label{item:1c}  
  If $V(x) =o(1)$ there exists $(\omega,\xi)\in S^{d-1}\times \R^d$ with
  \begin{align}\label{eq:12}
    Q(\xi+\i \sigma_\c \omega)=\lambda.
  \end{align}
\item \label{item:1cd}  If
\begin{align*}
&\forall \alpha: \partial^{\alpha}V_1(x) = o(|x|^{- |\alpha|}),\\
&V_2(x) = o(|x|^{-1/2}),
\end{align*}
then $\sigma_\c\in \Sigma_{\rm
  exc}(\lambda)$.
\end{enumerate}
\end{thm}

Theorem \ref{thm:12.7.2.7.9} \ref{item:1cd} gives stringent necessary
conditions on a decay rate, namely that it belong to
$\Sigma_{\rm exc}(\lambda)$.  To see that in certain situations \emph{all}
of the elements of $\Sigma_{\rm exc}(\lambda)$ can occur as decay rates, see the discussion in Subsection \ref{subsec:example} where we consider the family of $Q$'s which are polynomials in $\xi^2$.

In a generic sense (see the remark following the theorem), the next result gives a more precise estimate on the decay of $\phi$ once we know $\sigma_\c \in (0,\infty)$.

\begin{thm}\label{precise}
Suppose $0<\sigma_\c< \infty$.
Suppose  $\forall \alpha: \partial^{\alpha}V_1(x) = O(|x|^{-|\alpha| -
  \delta_1})$ and $V_2(x) = O(|x|^{-1/2 - \delta_2})$ with $\delta_1,\delta_2 >0$. Then
   either there exists $(\omega,\xi)\in S^{d-1}\times \R^d$
   satisfying 
 \begin{align}\label{eq:16eqn}
Q(\xi+\i \sigma_\c \omega)=\lambda\mand \nabla_\xi Q(\xi +\i \sigma_\c \omega)= 0, 
 \end{align}
  or  for any  $\epsilon\in (0,\epsilon')$ where  $\epsilon'=\min (\delta_1,
 2\delta_2, 1)$,
\begin{align*} 
 \e^{\sigma_{\c}(|x|-|
x|^{1-\epsilon})}\phi\in L^2.
\end{align*} 
\end{thm}

Note that \eqref{eq:12} is necessary for  $\sigma_\c\in \Sigma_{\rm
  exc}(\lambda)$ while \eqref{eq:16eqn} is sufficient. We give an example in
Subsection \ref{subsec:example} for which the $2d +2$ real equations \eqref{eq:16eqn}  (for $2d$ unknowns) do not
have solutions  in a  generic sense. Whence  for  that example  the second
alternative of Theorem \ref{precise} is generic.

The following theorem eliminates the possibility of super-exponential decay at the expense of rather strong decay assumptions on the potential:

\begin{thm}\label{thm:12.7.2.7.9b} 
Suppose $V_2(x) = O(|x|^{-q/2 -\delta})$ and 
$\partial^{\alpha}V_1(x) = O(|x|^{-(\delta+q +|\alpha|)/2})$,  $1\le
|\alpha| \le q$, where $\delta > 0$.  Then
  $\sigma_\c<\infty$ unless 
  $\phi=0$.

\end{thm}
 
\vspace{2mm} 

The restriction on the potentials in Theorem \ref{thm:12.7.2.7.9b} is in general not optimal.  In the special cases $Q(\xi) = \xi^2$ and $Q(\xi) = (\xi^2)^2$
we improve the bounds used to prove Theorem \ref{thm:12.7.2.7.9b} to
get better results in the next theorem. We prove this theorem in
Subsection \ref{subsection:Improvement for}.  There we use specific
properties of the above polynomials for which our verification of the bounds appears very ad hoc.  The main virtue of  Theorem \ref{thm:12.7.2.7.9b}
 is its generality.

\begin{thm} \label{better} Suppose $Q(\xi) = |\xi|^{2j}, j= 1$ or
  $j=2$. Suppose  $V_2(x) =
O(|x|^{-\delta - j/2})$ and    
$\partial^{\alpha} V_1(x) = O(|x|^{-(\delta + j +|\alpha|)/2})$, $1\le
|\alpha| \le j$,  where   $\delta >0$. Then $\sigma_\c<\infty$ unless 
  $\phi=0$.
\end{thm} 
In Lemma \ref{lemma:limit-meth-exampl} we show
 optimality of the  bound used to prove Theorem \ref{better} for
 $(\xi^2)^2$. Whence a possible further improvement of the decay rates
 specified for 
 $(\xi^2)^2$  would require  a
 completely new method of proof.
\begin{remarks}\label{remarks:results}
  \begin{enumerate}[1)]
\item For all our
    results  we can allow $V_2$ to be complex-valued virtually without any
    complication in the proofs, however we need 
$V_1$ to be real-valued and $\lambda$ to be real.
\item\label{item:CT}
If $\lambda \notin \Ran Q$ and $V(x)=o(1)$,  Theorems \ref{thm:start}  and
 \ref{thm:12.7.2.7.9}  give
\begin{align*}
 \sigma_c \ge \inf \{\sigma> 0| \,\exists (\omega,\xi) \in
 S^{d-1}\times \R^d: Q(\xi+\i \sigma \omega)=\lambda\}. 
\end{align*} There is a different proof of this bound using   a
Combes -Thomas argument \cite{CT}. 
\item \label{item:1bb} The result for $Q(\xi) = \xi^2$ in Theorem
  \ref{better} is well-known.  See for example \cite{FHH2O3, FHH2O2}.
\item \label{item:1cc} It is possible to treat some  variable coefficient
  cases for most of our results.  (We do not make these generalizations precise.)  This possibility  is due to the fact that  most of our results are based on the general
  theory of 
  pseudodifferential operators which is rather robust. Only Theorems
  \ref{thm:12.7.2.7.9b} and \ref{better} which are  based   on exact combinatorial formulas 
    do  not  readily generalize to variable coefficient cases.  In
    concrete situations as in Theorem \ref{better}  a perturbative argument
   works. Thus one can include  classes  of first and second order
   polynomials with variable
   coefficients for  the examples 
 $\xi^2$ and $(\xi^2)^2$. This follows readily by a little refinement
 of  the improved bounds, so-called subelliptic estimates, see \eqref{eq:26} and
 \eqref{eq:27}. 
\item \label{item:1ccBB }Another potential direction
 of generalization concerns elliptic real-analytic dispersion relations
 (rather than  elliptic polynomials), for example
 $Q(\xi)=(1+|\xi|^2)^{q/2}$ for any real 
 $q>0$.   Indeed we  expect that  our
 methods could yield versions of Theorems
 \ref{thm:start}--\ref{precise} for a  general class of such symbols. This
 would require (omitting further details)  a uniform analyticity
 radius say $\sigma_{\rm a}>0$ (in other words that there is an analytic extension of the
 symbol to a
 $d$-dimensional strip of width $2\sigma_{\rm a}$) and  conditions of   at most
 polynomial growth and ellipticity (with constants  being  locally uniform in the 
 imaginary part). Of course the hypothesis $\sigma_\c<\sigma_{\rm a}$ should be added
 to  the corresponding versions of Theorems \ref{thm:12.7.2.7.9} and
 \ref{precise}.  See \cite{MP2} for estimates leading to $\sigma_\c >
 0$ for a class of such dispersion relations under conditions
 overlapping the condition  \ref{item:2} of Theorem \ref{thm:start}.
\item \label{item:2g}
 One way to think about the
  conditions \eqref{eq:10} and \eqref{eq:11} is the following: Write
\begin{align*}
  Q(\xi+\i \sigma\omega(x))-\lambda=(X+\i Y)(\xi+\i \sigma\omega(x));\omega(x)=\hat x=x/|x|,
\end{align*} and look at
  the Hamiltonian $h(x,\xi)=X(\xi+\i \sigma\omega(x))$. Due to the Cauchy-Riemann
equations,  for any  Hamiltonian orbit $(x,\xi)$
  for $h$ (with $x\neq 0$)
  \begin{align}\label{eq:13}
  \begin{cases}
    \d_\tau \omega=|x|\dot \omega&=P_{\perp}(\omega)\partial_\xi X(\xi+\i \sigma\omega),\\
\d_\tau\xi= |x|\dot \xi&=\sigma P_{\perp}(\omega)\partial_\xi Y(\xi+\i \sigma\omega).
  \end{cases}
  \end{align} This is a reduced system of ordinary differential
  equations in the rescaled time $\tau$.
The conditions \eqref{eq:11} are exactly the conditions for a
fixed point of the flow  \eqref{eq:13}. In general $X(\xi+\i\sigma\omega)$ is constant  while $Y(\xi+\i
\sigma\omega)$ is growing for the flow  \eqref{eq:13}. 
\end{enumerate}
  \end{remarks}

\subsection{Example, $Q(\xi) - \lambda = G(\xi^2)$}\label{subsec:example}

The main object of this section is to gain some understanding of  the consequences of \eqref{eq:10} and  \eqref{eq:11}  for $\sigma > 0$.  For general $Q$ these two equations are actually $2 + 2(d-1) = 2d$ real scalar equations for the $2d$ variables $\xi, \omega, \mand \sigma$.  We take $\omega \in S^{d-1}$.  In the case at hand there is an overall rotational symmetry which implies that if $(\xi, \omega, \sigma)$ is a solution then for any real orthogonal matrix $R, (R\xi, R\omega, \sigma)$ is also a solution.  Let $z = (\xi + \i\sigma \omega)^2$.  We have two equations:
\begin{align}
G(z) = & 0,\\
P_{\perp}(\omega)\nabla_{\xi}Q(\xi +\i \sigma \omega) = & 2G'(z)P_{\perp}(\omega)\xi =0.
\end{align}
If $P_{\perp}(\omega)\xi =0$ then $ \xi = \pm |\xi| \omega$ so that
$G(z) = 0$ is the same as $G( (\pm|\xi| + \i\sigma)^2) =
0$.  Note that for each
pair of complex conjugate
 roots of $G$ there will generally  correspond two roots $\zeta = \pm|\xi| + \i \sigma$ in the upper half plane of the polynomial $\tilde G(\zeta): = G(\zeta^2)$.  On the
 other hand if  $P_{\perp}(\omega)\xi \ne 0$ then we have the two
 equations $G(z) = G'(z) = 0$ which require $G$ to have a multiple
 zero.  If $Q$ has degree $q$ this can only happen at $ \le \frac{q
   -2}{2} $ values of  $ \lambda$. If $\lambda$ is not one of these at most  $ \frac{q
   -2}{2} $ possible real numbers there are $ \le q/2$ exceptional numbers.   In the case of $G( z ) = z^2 - \lambda$,
 involving the bilaplacian, if $ \lambda \ne 0$ there is exactly one
 solution with positive $\sigma$.  Namely $\sigma = \lambda^{1/4}$ if
 $\lambda > 0$ and $\sigma = (-\lambda/4)^{1/4}$ if $\lambda < 0$. On
 the other hand by 
 the construction below each of these cases is realized by a compact
 support potential.

In the situation of Theorem \ref{precise}  where
we have $Q(\xi + \i\sigma\omega) = \lambda \mand \nabla Q(\xi +
\i\sigma\omega) = 0$,  either we are in the non-generic case $G(z) =
G'(z) = 0$ or we have $\xi + \i\sigma \omega = 0$, an impossible
situation since $\sigma > 0$.

We remark that as in the $N$-body problem where for a fixed negative
eigenvalue there are typically many exponential decay rates possible
determined by the thresholds for the problem,  we have a similar
situation for $Q(p) + V(x)$ even for $V$ of compact support.  

Let us for any  $\lambda \in \R$ consider  any  zero $z\in
\R\setminus [0,\infty)$ of the function $G$. There is a unique
$k\in\C$ with $\sigma:=\Im k>0$ and $z=k^2$, and according to the
above discussion 
 $\sigma \in \Sigma_{\rm exc}(\lambda)$.  We will display a real $V \in
C_\c^{\infty}(\mathbb{R}^d)$ and a
function $\phi\in L^2$ satisfying $(G(-\Delta) + V) \phi = 0$  with critical
decay rate $\sigma_\c = \sigma$.   Letting  $K$ be the integral kernel of
$(-\Delta - z)^{-1}$ the function $\tilde\phi(x) = \text{Re}K(x,0)$
satisfies $G(-\Delta) \tilde\phi = 0$ for $x \ne 0$.  Since
$\tilde\phi(x) > 0$ for small $|x|$ (see \cite[p. 360]{AS} and
\cite[p. 232--233]{T}) we can modify  $\tilde \phi$ there to obtain a
function $\phi$ with $V: = G(-\Delta)\phi/\phi$ smooth with compact
support.  To carry out this modification choose $R$ so that $
\tilde\phi > 0$  for $0 < |x| < R$.  Let $ \chi \in
C_\c^{\infty}(\mathbb{R}^d)$ with $0 \le \chi \le 1$ and $\chi = 1$ if
$0 \le |x| < R/2 \mand \chi = 0$ if $ |x| > 3R/4$.  Let $ \phi = \chi
+ (1-\chi) \tilde \phi$.   Then $\phi$ is smooth and positive for $|x|
< R$ while $G(-\Delta)\phi$ has support in $|x| \le 3R/4$.  Thus
indeed $V$
is real, smooth and  compactly supported.  Clearly $\sigma_\c = \sigma$ (see
\cite[p. 364]{AS} and \cite[p. 232--233]{T}).

These considerations show that for $Q$ any elliptic polynomial in
$\xi^2$, except  for at most $(q-2)/2$ real $\lambda$'s, $\sigma \in
\Sigma_{\rm exc}(\lambda)$ is sufficient for $\sigma$ to be the  critical decay rate
 of an eigenfunction with eigenvalue $\lambda$ and for some  compactly supported real potential $V$. The converse statement, necessity,  is
stated for a more general $Q$ in Theorem \ref{thm:12.7.2.7.9} \ref{item:1cd}.


\subsection{Notation and calculus considerations}\label{subsec:Notation} We shall use
the Weyl calculus for symbols in  $S(m,g)$ where the weight $m$ may vary but the
metric   will be
\begin{align*}
  g=\inp{x}^{-2}\d x^2+\inp{\xi}^{-2}\d \xi^2.
\end{align*} Here and henceforth $\inp{x}=(1+|x|^2)^{1/2}$ and similarly with $x\to
\xi$; let also $\hat x= x/|x|$.  See \cite[Chapt. XVIII]{Ho} for an exposition of the Weyl
calculus (including the results stated below). Recalling  $q=\text{degree}(Q)$ obviously $Q\in
S(\inp{\xi}^q,g)$. 

\begin{subequations}
Recall the $L^2$-boundedness result, here amounting to
\begin{align}
  \label{eq:8L2B}
  a\in S(1,g) \Rightarrow   \|\Opw(a)\|\leq C,
\end{align} where $C\geq 0$ can be chosen independently of $a$ from
any bounded family of symbols in  $S(1,g)$ (i.e. any family for which
each semi-norm has a uniform bound). A family of such symbols
is  said to be uniformly in $S(1,g)$. We shall use similar terminology  
for symbols in $S(m,g)$.

By the Fefferman-Phong inequality
\begin{align}
  \label{eq:8FP}
  a\in S(\inp{x}^{2}\inp{\xi}^{2},g), \,a\geq 0\Rightarrow   \Opw(a)\geq -C,
\end{align} where $C\geq 0$ can be chosen independently of $a$ from
any bounded family of such symbols. 
 \end{subequations}

Let $\chi_-$ and $\chi_+$ denote smooth non-negative functions on $\R$ with
$\chi_-(t)=1$ for $t\leq 1$, $\chi_-(t)=0$ for $t\geq 2$ and 
\begin{align}\label{eq:8quad}
  \chi_-^2+ \chi_+^2=1.
\end{align}
  For any $\kappa>0$ we define $\chi_-(t\leq
\kappa)=\chi_-(t/ \kappa)$ and $\chi_+(t\geq
\kappa)=\chi_+(t/ \kappa)$.

In this paper bounding constants are typically denoted by $c$,  $C$ or $C_j$. They may
vary from line to line. For any operator $A$ (or form) we abbreviate
the inner product 
 $\inp{\psi,A\psi}=\inp{A}_\psi$.
\subsubsection{Distorted $|x|$}\label{subsubsec:Distorted} 
We are going to use two qualitatively different
distorted versions of the function $|x|$ on $\R^d$. The first one is
\begin{subequations}
\begin{align}
  \label{eq:r_1}
  r(x)=r_1(x)=\inp {x}.
\end{align} The second one is given in terms of a parameter
$\epsilon\in(0,1)$ as 
\begin{align}
  \label{eq:r_epsilon}
  r(x)=r_\epsilon(x)=\inp {x}-\inp {x}^{1-\epsilon}+1.
\end{align}  
\end{subequations} Note that these are positive smooth  strictly convex functions  that
at infinity  behave as  $|x|$. However while $r_1$ tends to be
degenerately convex at infinity (just as   $|x|$),  this deficiency 
appears somewhat  cured for $r_\epsilon$ (in particular in the regime
where $\epsilon$ is small). In both cases we shall use the notation  $\omega=\omega(x)=\grad
r$.   The functions $r_{\epsilon}$  were used in \cite{RT} in a different context.  See also \cite{IS1}.

\section{Ideas of proof of Theorems  \mbox{\rm \ref{thm:12.7.2.7.9}} and \ref{precise}}\label{sec:Ideas of proofs}
Consider the  function $r$ given by either \eqref{eq:r_1}
or \eqref{eq:r_epsilon}.

We introduce for  $\sigma\geq 0$
\begin{align*}
  Q(\xi+\i \sigma\omega(x))-\lambda=(X+\i Y)(\xi+\i \sigma\omega(x)).
\end{align*} 
This is to  leading order the  symbol of
\begin{align*}
  \e^{\sigma r}\{Q(p)-\lambda\}\e^{-\sigma r}=Q(p+\i \sigma \omega(x))-\lambda.
\end{align*}

Also we introduce 
the
distorted energy surface 
\begin{align}\label{eq:15}
S_{\sigma,\lambda}=\{(x,\xi)\in \R^d\times \R^d| Q(\xi+\i \sigma\omega(x))=\lambda\}=\{X= Y=0\}.
\end{align} 

 Now suppose $\phi$ is an eigenfunction with eigenvalue $\lambda$,
i.e. $(H-\lambda)\phi=0$. Suppose that $\phi_\sigma:=\e^{\sigma r}\phi\in L^2$ for some
small $\sigma>0$ (We do not justify this assumption here.  It is
proved in Section \ref{sec:start} under the assumptions of Theorem
\ref{thm:start}.)  So we consider an eigenfunction $\phi$ with
$\phi_{\sigma_0}\in L^2$ for some  $\sigma_0>0$  and want to derive a priori
bounds of $\phi_{\sigma}$ for $\sigma< \sigma_\c$.  See \eqref{eq:3}
for an example of such bound.

The result Theorem \ref{thm:12.7.2.7.9} \ref{item:1c} can be interpreted as an energy estimate which will not
 be discussed here. On the other hand Theorem  \ref{thm:12.7.2.7.9}
 \ref{item:1cd} and Theorem   \ref{precise}
 are strongly based on (strict) positivity of a certain
commutator to be explained. 

Introducing the  shorthand notation
\begin{align*}
  \eta=\sigma \omega(x)\mand \zeta =\xi+\i \eta,
\end{align*} the Cauchy-Riemann equations for $\zeta\to Q(\zeta)$ and the
chain rule allow us to  calculate the
Poisson bracket
\begin{align}\label{eq:16pois}
  \sigma^{-1}\{X,Y\}=\partial_\xi X\omega'(x)\partial_\xi^T X+\partial_\xi Y\omega'(x)\partial_\xi^T Y.
\end{align}
Whence in particular
\begin{align}
  \label{eq:6}
  \{X,Y\}\geq 0.
\end{align}

We propose to consider the {\it conjugate operator}  $A$ with Weyl symbol
\begin{align*}
  a=rY(\zeta)=rY(\xi+\i \sigma\omega(x)).
\end{align*}

Consider 
\begin{align*}
  \i[Q(p),\e^{\sigma r}A\e^{\sigma r}]=\e^{\sigma r}\parb{\i[\tilde
    X,A]+2\Re(\tilde YA)}e^{\sigma r},
\end{align*} where  $\tilde X=\Re\parb{\e^{\sigma
    r}Q(p)\e^{-\sigma r}}-\lambda$ and  $\tilde Y=\Im\parb{\e^{\sigma r}Q(p)\e^{-\sigma r}}$.

To leading order the symbol of the operator between exponentials to the
right has symbol
\begin{align*}
  r\{X,Y\}+2rY^2+\{X,r\}Y.
\end{align*} Note that the second term is also non-negative.

In the rest of our discussion in this section we only discuss the proof of Theorem \ref{precise}.

Let us note that a sufficient condition for $S_{\sigma, \lambda}$   to be  a codimension  $2$ submanifold of
$\R^{2d}$ is 
\begin{align}
  \label{eq:15b}
  \nabla_\xi Q(\xi +\i \sigma\omega(x))\neq 0 \text{ for all } (x,\xi)\in S_{\sigma, \lambda}.
\end{align} This is due to the Cauchy-Riemann equations.

Under the
regularity condition \eqref{eq:15b}  there is a chance
of deriving some  positivity from \eqref{eq:16pois}, gaining
positivity from the first term. In fact, discussing
here  only $r=r_\epsilon$, a slight strengthening of the
regularity condition \eqref{eq:15b} 
yields  the following bounds in a neighbourhood of  $S_{\sigma, \lambda}$,
\begin{align*}
  r\{X,Y\}&\geq 3cr^{-\epsilon},\\
\{X,r\}Y&\geq - rY^2-C r^{-1},
\end{align*} 
and therefore also
\begin{align}\label{eq:1}
  r\{X,Y\}+2rY^2+\{X,r\}Y\geq 2cr^{-\epsilon}+rY^2.
\end{align} 

Next by using a proper energy
cut-off and  \eqref{eq:8FP} we show that \eqref{eq:1}
is preserved under quantization and, more precisely,  we derive a bound like
\begin{align}
  \label{eq:2}
  c\|r^{-\epsilon/2}\phi_\sigma\|^2\leq -\inp{\i[V,A]}_{\phi_\sigma}-\|r^{1/2}\tilde Y\phi_\sigma\|^2+C\|\phi\|^2.
\end{align} Now we insert  the splitting $V=V_1+V_2$ of Theorem \ref{precise}
 into the first term to  the right. Assuming $\epsilon<\delta_1$ we
can estimate the contribution from $V_1$ by doing the
commutator. Assuming also $\tfrac 12(1+\epsilon)<1/2 +\delta_2$ we
can estimate the contribution from $V_2$ by undoing the
commutator using the second term to  the right. Thus we have converted  \eqref{eq:2} into  
\begin{align}
  \label{eq:3}
  \|r^{-\epsilon/2}\phi_\sigma\|^2\leq C\|\phi\|^2.
\end{align} 
Finally   by taking  $\sigma\nearrow \sigma_\c$ in \eqref{eq:3} we
 obtain Theorem \ref{precise}.

\section{Elaboration on energy localization}\label{sec:energy}
This section contains various preliminary bounds valid for any eigenfunction
$\phi$, $(H-\lambda)\phi=0$, with $\sigma_\c>0$.  These are, more precisely,
a priori bounds on  $\phi_\sigma=\e^{\sigma r}\phi$ for $\sigma\in [0,
\sigma_\c)$. Here $r$ is the  function given by either \eqref{eq:r_1}
or \eqref{eq:r_epsilon}. The
bounds are uniform in $\sigma$ varying in any bounded subset of $[0,
\sigma_c)$. In the
last subsection we establish some  uniform bounds for
 certain non-convex  parameter-dependent approximations. In this
 section we shall only need boundedness of $V$ (i.e. decay assumptions
 will not be
 needed).

 \subsection{Sobolev regularity}\label{subsec:Sobolev regularity} Let
 us note that 
  \begin{align}
   \label{eq:16sob}
   \forall \sigma\in [0,
\sigma_\c),\;s\in \R: \;\phi_\sigma\in r^{-s}H^q,
  \end{align} where $H^q$ is the standard Sobolev space of order $q$,
  \begin{align*}
    H^q=\{ \psi\in L^2|\partial^\alpha\psi\in L^2\text{ for
    }|\alpha|\leq q\}.
  \end{align*} In fact letting $f(r)=\sigma r+s\ln r$ we can write
  \begin{align*}
    \e^{f}Q(p)\e^{-f}=\Opw(a),
  \end{align*} where $a$ is an elliptic symbol  in $S(\inp{\xi}^q,g)$,
  cf. Appendix \ref{sec:appA}. Then \eqref{eq:16sob} follows from
  \begin{align*}
    \Opw(a)r^{s}\phi_\sigma=r^{s}\e^{\sigma r} (\lambda-V)\phi\in L^2.
  \end{align*}

\subsection{Energy bounds}\label{subsec:energy bounds}
For $s\in \R$ we calculate on the one hand
\begin{subequations}
  \begin{align}\label{eq:4}
  \|r^s(\tilde X+
\i \tilde Y)\phi_\sigma\|^2=\|r^s\tilde X \phi_\sigma\|^2+\|r^s\tilde
Y \phi_\sigma\|^2-2\Im \inp{\tilde X r^{2s}\tilde Y}_{\phi_\sigma},
\end{align} while on the other hand
\begin{align}\label{eq:5}
  \|r^s(\tilde X+
\i \tilde Y)\phi_\sigma\|^2=\|r^sV\phi_\sigma\|^2.
\end{align} 
\end{subequations}
Whence to get a useful bound we need to examine the last
term to  the right in  \eqref{eq:4}:
\begin{align*}
  &-2\Im \inp{\tilde X r^{2s}\tilde Y}_{\phi_\sigma}\\&\geq\inp{\i [\tilde
    X ,\tilde Y]}_{r^s\phi_\sigma}-\tfrac15 (\|r^s\tilde
  Y\phi_\sigma\|^2+\|r^s\tilde X\phi_\sigma\|^2)
-C_1\|r^{s-1}\phi_\sigma\|^2-C_1\|r^{s-1}\tilde X \phi_\sigma\|^2\\&\geq\inp{\i [\tilde
    X ,\tilde Y]}_{r^s\phi_\sigma}-\tfrac14 (\|r^s\tilde
  Y\phi_\sigma\|^2+\|r^s\tilde X\phi_\sigma\|^2)
-C_2\|r^{s-1}\phi_\sigma\|^2-C_2\|V\phi\|^2.
\end{align*} Here we used the ellipticity of $Q(p)$ and $ \tilde X $
estimating, cf. \eqref{eq:8FP},  like
\begin{align}
  \label{eq:ellip}
  r^t\inp{p}^{2q} r^t\leq C_{1,t}\tilde X r^{2t}\tilde X+C_{2,t}r^{2t}.
\end{align}

Using \eqref{eq:8FP}, \eqref{eq:6} and \eqref{eq:ellip} we obtain
\begin{align*}
&\inp{\i [\tilde
    X ,\tilde Y]}_{r^s\phi_\sigma}\\ &\geq -C_1(\|r^{s-1}\tilde
  X\phi_\sigma\|^2+\|r^{s-1}\phi_\sigma\|^2)\\ 
\\ &\geq -\tfrac14 \|r^s\tilde X\phi_\sigma\|^2
-C_2\|r^{s-1}\phi_\sigma\|^2-C_2\|V\phi\|^2.
\end{align*} In combination with \eqref{eq:4} these bounds yield
\begin{align*}
  \|r^s(\tilde X+ 
\i \tilde Y)\phi_\sigma\|^2\geq\tfrac12 (\|r^s\tilde
  Y\phi_\sigma\|^2+\|r^s\tilde X\phi_\sigma\|^2)
-C\|r^{s-1}\phi_\sigma\|^2-C\|V\phi\|^2,
\end{align*}  and therefore by  \eqref{eq:5} that
  \begin{align}\label{eq:7a}
  \|r^s\tilde  Y\phi_\sigma\|^2+\|r^s\tilde X\phi_\sigma\|^2\leq 2\|r^sV\phi_\sigma\|^2+C\|r^{s-1}\phi_\sigma\|^2+C\|V\phi\|^2.
\end{align} 
\subsection{Parameter-dependent  bounds}\label{subsec:Parameter
  dependent   bounds}  We let $r=r_1$. Rather than considering $\e^{\sigma r}$ we now look at
$\e^{f_m}$ where
\begin{align}\label{eq:approx}
  f_m=f_m(r)
=r(\sigma+\gamma/(1+r/m));  m\in \N. 
\end{align} We have in mind to use this construction for $\sigma\in
[0,\sigma_\c)$ and small $\gamma>0$.
Clearly
$\e^{f_m}\to \e^{(\sigma +\gamma)r}$ for $m\to \infty$. Since
$f''_m(r)=-\tfrac {2\gamma}r\tfrac {r/m}{(1+r/m)^3}<0$ we are lacking
the convexity property if we replace  $\e^{\sigma r}\to \e^{f_m}$ in
Section 
\ref{sec:Ideas of proofs}. In fact we need to modify \eqref{eq:6}, for
example as 
\begin{align*} 
   \{X,Y\}&\geq -\tfrac {2 \gamma}r \parb{(\partial_\xi X\cdot
    \omega(x))^2+(\partial_\xi Y\cdot \omega(x))^2}\\&\geq -\gamma \tfrac
  {C}r \parb{X^2+1}.
\end{align*} Here $X$ and $Y$ are defined in terms of $f_m$ (rather
than in terms of the exponent $\sigma r$ as before), $\omega(x)=\grad r$, and $C>0$ is independent of $m$
and $\gamma$. 

Using this bound we obtain the following  modifications of
\eqref{eq:7a}, using the (slightly inconsistent) notation $\phi_m=\e^{f_m}\phi$:
  \begin{align}
    \begin{split}
  \label{eq:16a}  \|r^s\tilde
  Y\phi_m\|^2+\|r^s\tilde X\phi_m\|^2&\leq2\|r^sV\phi_m\|^2+C\|r^{s-1/2}\phi_m\|^2+C\|V\phi\|^2.
 \end{split}
\end{align}
 The constants can be chosen to be independent of  $m\in \N$
and  $\gamma\in (0,1]$.  We also note
that \eqref{eq:16a} is valid under the assumption $r^s\e^{f_m}\phi\in
L^2$ only, in particular without  assuming the strict inequality 
$\sigma<\sigma_\c$. (This will be needed in Section
 \ref{sec:start}). Somewhat refined  $\gamma$-dependent
constants can be given, however this will not be needed. 

\section{Proof of Theorem \mbox{\rm \ref{thm:12.7.2.7.9}
  \ref{item:1c}}}\label{sec:Proof of Theorem one}
Suppose \eqref{eq:12} does not have any solution, i.e.
\begin{align*}
  \inf_{\omega,\xi}|Q(\xi+\i \sigma_\c\omega)-\lambda|^2\geq 2\kappa>0.
\end{align*}
Then we fix
$\sigma<\sigma_\c$, slightly smaller, and small $\gamma>0$ with
$\sigma +\gamma>\sigma_\c$. Define  $X$ and $Y$ in terms of the
approximation $f_m$ considered in Subsection \ref{subsec:Parameter
  dependent   bounds}. We arrive at a contradiction by showing  the
uniform bound
\begin{align}
  \label{eq:17}
  \|\phi_m\|^2\leq C\|\phi\|^2;\;\phi_m=\e^{f_m}\phi.
\end{align} 

Indeed with a proper adjustment of $\sigma$ and $\gamma$ we have
$X^2+Y^2\geq \kappa$ for $|x|$ large (by a continuity argument). Then
by using \eqref{eq:8FP}, 
\eqref{eq:ellip} and 
\eqref{eq:16a}  with $s=0$  we obtain
\begin{align}\label{eq:9T}
  \begin{split}
   \|\phi_m\|^2&\leq C_1\parb{ \|(\tilde X^2+\tilde
    Y^2)^{1/2}\phi_m\|^2+\|r^{-1}\phi_m\|^2}\\
&\leq C_2\parb{\|V\phi_m\|^2+\|r^{-1/2}\phi_m\|^2}.
\end{split}
\end{align} Clearly \eqref{eq:17} follows from \eqref{eq:9T}.

\section{Proof of Theorem \mbox{ \rm \ref{precise}}}\label{sec:Bounding with deformed r}
Let $r=r_\epsilon$, and let $\omega=\omega(x)=\grad
r$. Using Section \ref{sec:energy} we shall  give the missing details in the outline of proof of Theorem
\ref{precise} given in Section \ref{sec:Ideas of
  proofs}. So suppose that the equations \eqref{eq:16eqn} do not have solutions. We look at the
state $\phi_\sigma=\e^{\sigma r}\phi$
for $\sigma< \sigma_\c$, but close, and want to prove \eqref{eq:3}.

We introduce 
\begin{align}\label{eq:15_hat}
\hat S_{\sigma, \lambda}=\{(x,\xi)\in\parb{\R^d\setminus\{0\}}\times \R^d| Q(\xi+\i \sigma\hat x)=\lambda\},
\end{align} which at infinity is close to the set 
$S_{\sigma, \lambda}$ of \eqref{eq:15}.  We localize near $\hat S_{\sigma, \lambda}$ introducing a quantization, say
$\tilde\chi_-$, of the symbols
$\chi_-=\chi_-(X^2+Y^2\leq \kappa)$, $\kappa>0$ small, see Subsection
\ref{subsec:Notation}. We have (slightly
inconsistently) shortened the
notation suppressing the dependence of $\kappa$. Note in particular
that we are going to use the decomposition
of unity \eqref{eq:8quad} for the $\kappa$-dependent partition functions.

Now we  estimate using
\eqref{eq:8FP}, \eqref{eq:ellip} and  \eqref{eq:7a}
\begin{subequations}
\begin{align}\label{eq:9}
  \begin{split}
   \|\tilde\chi_+\phi_\sigma\|^2&\leq C_1\parb{ \|(\tilde X^2+\tilde
    Y^2)^{1/2}\phi_\sigma\|^2+\|r^{-1}\phi_\sigma\|^2}\\
&\leq C_2\parb{\|V\phi_\sigma\|^2+\|r^{-1}\phi_\sigma\|^2}.
\end{split}
\end{align} Using   \eqref{eq:ellip}, \eqref{eq:7a}  and 
\eqref{eq:9} we obtain the improvement 
\begin{align}\label{eq:9bb}
  \begin{split}
   &\|\inp{p}^q\tilde\chi_+\phi_\sigma\|^2\\&\leq
   C_1\|\tilde X \tilde\chi_+\phi_\sigma\|^2+
   C_2\|\tilde\chi_+\phi_\sigma\|^2\\
&\leq
   C'_1\|\tilde X
   \phi_\sigma\|^2+C'_2\|r^{-1}\phi_\sigma\|^2+C_3\|Vr^{-1}\phi_\sigma\|^2\\&\quad
   +
   C_2\|\tilde\chi_+\phi_\sigma\|^2\\
&\leq C\parb{\|V\phi_\sigma\|^2+\|r^{-1}\phi_\sigma\|^2}.
\end{split}
\end{align}   
\end{subequations}

Redoing the considerations in Section \ref{sec:Ideas of proofs}  we see 
that 
\eqref{eq:1} is modified as
\begin{align*}
  &r\{X,Y\}+2rY^2+\{X,r\}Y\\&\geq r\{X,Y\}+rY^2-C_1r^{-1}\inp{\xi}^{2q}\\& \geq 2cr^{-\epsilon}+rY^2-C_2\inp{\xi}^{2q}\chi^2_+.
\end{align*} for small enough $\kappa$.  We quantize, combine with \eqref{eq:9bb} and conclude
\begin{align*}
  2c\|r^{-\epsilon/2}\phi_\sigma\|^2\leq -\inp{\i[V,A]}_{\phi_\sigma}-\|r^{1/2}\tilde Y\phi_\sigma\|^2+C_1\|V\phi_\sigma\|^2+C_2\|r^{-1/2}\phi_\sigma\|^2+C_3\|\phi\|^2.
\end{align*} Whence using that $V=O(|x|^{-\delta})$,
$\delta=\min(\delta_1,\delta_2)$, 
we obtain 
  the  bound  \eqref{eq:2} for $\epsilon<
  \min(2\delta_1,2\delta_2,1)$, that is 
\begin{align}
  \label{eq:2BB}
  c\|r^{-\epsilon/2}\phi_\sigma\|^2\leq -\inp{\i[V,A]}_{\phi_\sigma}-\|r^{1/2}\tilde Y\phi_\sigma\|^2+C\|\phi\|^2.
\end{align} Using the conditions on $V_1 \mand V_2$ and \eqref{eq:2BB},
assuming also $\epsilon<\min(\delta_1, 2\delta_2)$, and  again
\eqref{eq:ellip} and \eqref{eq:7a} we can show  \eqref{eq:3} as
follows: We estimate
\begin{align*}
  -\inp{\i[V_1,A]}_{\phi_\sigma}&\leq C_1\|\inp{p}^qr^{-\delta_1/2}\phi_\sigma\|^2\leq C_2\|r^{-\delta_1/2}\phi_\sigma\|^2,\\ 
-\inp{\i[V_2,A]}_{\phi_\sigma}&\leq C_1\|r^{-\delta_2}\phi_\sigma\|^2+\|r^{1/2}\tilde Y\phi_\sigma\|^2+C_2\|r^{-1/2}\phi_\sigma\|^2,
\end{align*} and insert into \eqref{eq:2BB} leading to the uniform bound
\begin{align*}
  \tfrac c2\|r^{-\epsilon/2}\phi_\sigma\|^2\leq \tilde C\|\phi\|^2.
\end{align*} Now take $\sigma\nearrow \sigma_\c$.

\section{Proof of Theorem \mbox {\rm
  \ref{thm:12.7.2.7.9} \ref{item:1cd}}}\label{sec:Bounding with rapprox
  |x|}

Let $r=r_1$, and let $\omega(x)=\grad
r$. Note that $\omega'(x)=r^{-1}P_{\perp}(\omega(x))$. We
suppose  that the conditions  \eqref{eq:10} and \eqref{eq:11} are 
not  both true at any point $(\omega,\xi)$  and want to find a 
contradiction. For that we look at the state $\phi_\sigma=\e^{\sigma r}\phi$
where $\sigma< \sigma_\c$, but close. As a first step (serving mainly as
a warm up) we are heading  toward the following
analogue of \eqref{eq:3} which we will show to be uniform in  $\sigma< \sigma_\c$:
\begin{align}
  \label{eq:3g}
  \|\phi_\sigma\|^2\leq C\|\phi\|^2.
\end{align} Let $\hat S_{\sigma, \lambda}$ be defined by \eqref{eq:15_hat}. By continuity
 and compactness we have uniformly  in $\sigma$ close to $\sigma_\c$,
\begin{align}
  \label{eq:15bg}
\begin{split}
  &\nabla_\xi X(\xi +\i \sigma\hat x)P_{\perp}(\hat x)\nabla_\xi
  X(\xi +\i \sigma\hat x)\\ &+\nabla_\eta X(\xi +\i
  \sigma\hat x)P_{\perp}(\hat x)\nabla_\eta X(\xi +\i
  \sigma\hat x)\geq k>0\\
   &\text{ for all points
  in } (x,\xi)\in \hat S_{\sigma, \lambda}.
\end{split} 
\end{align} In fact    this 
 is valid in  a neighbourhood of $\hat S_{\sigma_\c, \lambda}$, and  we can also freely
replace $\hat x$ by $\omega(x)$ in \eqref{eq:15bg} (since
only large $|x|$ matters below).
 In particular we have the following version
of \eqref{eq:1} in a neighbourhood of $S_{\sigma, \lambda}\cap\{r>R\}$ for $R>1$
sufficiently large:
\begin{align*}
  r\{X,Y\}+2rY^2+\{X,r\}Y\geq 2c+rY^2.
\end{align*} 

Using \eqref{eq:ellip}, \eqref{eq:7a}  and the procedure of Section
\ref{sec:Bounding with deformed r} we obtain, cf. \eqref{eq:9bb},
\begin{align*}
  \begin{split}
   \|\inp{p}^q\tilde\chi_+\phi_\sigma\|^2
\leq C\parb{\|V\phi_\sigma\|^2+\|r^{-1}\phi_\sigma\|^2}.
\end{split} 
\end{align*} We continue mimicking  Section
\ref{sec:Bounding with deformed r} and
come to the  following analogue of \eqref{eq:2BB} (using now  only
that $V=o(|x|^0)$)
\begin{align*}
  c\|\phi_\sigma\|^2\leq
  -\inp{\i[V,A]}_{\phi_\sigma}-\|r^{1/2}\tilde Y\phi_\sigma\|^2+C
\|\phi\|^2,
\end{align*} and from this  indeed \eqref{eq:3g}. Next we take
$\sigma\nearrow \sigma_\c$, and we
conclude that $\e^{\sigma|x|}\phi\in L^2$ with   $\sigma= \sigma_\c$.

The next step is to use Subsection \ref{subsec:Parameter
  dependent   bounds} to show that 
$\e^{\sigma|x|}\phi\in L^2$ with   $\sigma$ 
slightly bigger, which obviously    is a contradiction. We mimic 
Section \ref{sec:Proof of Theorem one}. So fix $\sigma<\sigma_\c$, slightly smaller, and small $\gamma>0$ with
$\sigma +\gamma>\sigma_\c$. Define  $X$ and $Y$ in terms of the
approximation $f_m$ considered in Subsection \ref{subsec:Parameter
  dependent   bounds} and consider the corresponding state
$\phi_m=\e^{f_m}\phi$.  Using  \eqref{eq:ellip}  and 
\eqref{eq:16a}  with $s=0$ mimicking Section
\ref{sec:Bounding with deformed r} we obtain the following version  of \eqref{eq:2BB}
\begin{align*}
  \begin{split}
  c\|\phi_m\|^2\leq
  -\inp{\i[V,A]}_{\phi_m}-\|r^{1/2}\tilde
  Y\phi_m\|^2 +C\|\phi\|^2.
\end{split}
\end{align*} Estimating the commutator as before  we arrive at  
the estimate
\begin{align*}
  \|\phi_m\|^2\leq C\|\phi\|^2.
\end{align*} We obtain a contradiction by letting $m\to \infty$.

\section{Proof of Theorem \ref{thm:start}}\label{sec:start}
If $\lambda \notin \Ran Q$ and $V(x) = o(1)$ at infinity (so that $V$ is a relatively compact perturbation of $Q(p)$) we can use the Combes-Thomas method \cite{CT} to see that $\sigma_c > 0$.  We omit the details.  

Otherwise take $\sigma=0$ and $r=r_1$ in \eqref{eq:approx}. Whence we consider $f_m:= \gamma
r(1 +  r/m)^{-1}$, $m\in \N$. Let the number $\gamma\in (0,1]$ be taken
small. Under the conditions \ref{item:2} of  Theorem \ref{thm:start} and abbreviating $\phi_m=\e^{f_m}\phi$ we shall prove the estimate 
\begin{align}\label{eq:16basicest}
  \|\phi_m\|^2\leq C\|\phi\|^2,
\end{align} for small enough $\gamma$. Taking $m\to \infty$ completes
the proof of Theorem \ref{thm:start}.

Clearly
\begin{align*}
\nabla f_m(x) = \gamma g_m(r) \omega(x) ;  g_m:= (1 + r/m)^{-2},\;\omega = \nabla r. 
\end{align*} We consider the symbols $X_m$ and $Y_m$  given by
conjugation by $\e^{f_m}$ (in agreement with the previous
sections). We are going to construct a conjugate operator. The
previous sections suggest the quantization of $rY_m$, however we
prefer to use this symbol with a different normalization: Letting 
\begin{align*}
a_m(x,\xi) = (r/\gamma g_m(r))Y_m(x,\xi) 
\end{align*} we note that $a_m\in S(r\inp{\xi}^q,g)$ uniformly in $m$
(see Subsection \ref{subsec:Notation} for terminology). Introduce partition functions
$\chi_{-,m}=\chi_-(X_m^2+Y_m^2\leq \kappa)$ and
$\chi_{+,m}=\chi_+(X_m^2+Y_m^2\geq \kappa)$ for (small)
$\kappa>0$. Indeed we can 
estimate for  small $\kappa, \gamma>0$ using the assumptions
\begin{align*}
  \{X_m,a_m\}+2a_mY_m \geq 2c-C\inp {\xi}^{2q}\chi_{+,m}^2
\end{align*} for  constants $c,C>0$ being independent of $m$. Here we
use  that 
\begin{align*}
  \{X_m,a_m\}=|\nabla Q(\xi)|^2 + \gamma  a_{0,\gamma,m}+   a_{-1,\gamma,m},
\end{align*} where $  a_{0,\gamma,m}\in S(\inp{\xi}^q,g)$ and $
a_{-1,\gamma,m}\in S(r^{-1}\inp{\xi}^q,g)$,  both uniformly in $\gamma$ and
$m$,  and we use that $2a_mY_m\geq0$. We
quantize  yielding the bound (dropping indices) 
\begin{align}\label{eq:18bnd}
  \Im \parb {A(\tilde X+\i\tilde Y)}\geq c - \Re\parb {B(\tilde X+\i\tilde Y)}-Cr^{-1/2}\inp{p}^{2q}r^{-1/2},
\end{align} where the symbol of $B$ is in $S(\inp{\xi}^q,g)$ uniformly in 
$m$, and the constants $c, C>0$ are  independent of 
$m$,  cf. Appendix \ref{sec:appA}.

Now we apply \eqref{eq:18bnd} to the state $\phi_{m,n}:=\chi_-(r\leq
n)\phi_{m}$. Taking  $n\to \infty$ by using  Lebesgue's dominated
convergence theorem leaves us with
\begin{align*}
  \inp{ \tfrac \i2 [A,V_1]}_{\phi_m}-\Im \inp{r^{-1}A\phi_m,
    rV_2\phi_m}-\Re \inp{B^*\phi_m,
    V\phi_m}\geq \inp{c-Cr^{-1/2}\inp{p}^{2q}r^{-1/2}}_{\phi_m}.
\end{align*} Next we use the Cauchy-Schwarz inequality yielding the
following bound for any (small) $\varepsilon>0$
\begin{align*}
  C_{\varepsilon}  \|\inp{p}^{q}\phi\|^2 \geq \tfrac 12 c \|\phi_m\|^2- \varepsilon
    \|\inp{p}^{q}\phi_m\|^2.
\end{align*} 
 Finally we invoke \eqref{eq:ellip} and \eqref{eq:16a} with 
 $t=s=0$  yielding \eqref{eq:16basicest}.

\section{Proof of
Theorems \ref{thm:12.7.2.7.9b} and \ref{better}}\label{sec:Absence}
Let $r=r_\epsilon$, and let $\omega=\omega(x)=\grad
r$.  Defining 
\begin{align}
  \label{eq:48}
  a=(a_1,\dots, a_d)=\e^{-\sigma r}p\e^{\sigma r}=p-\i\sigma \omega,
\end{align}  
consider 
\begin{align*}
  \e^{-\sigma r}Q(p)\e^{\sigma r}=Q(p-\i \sigma \omega) = Q(a).
\end{align*}
For the proof of Theorem \ref{thm:12.7.2.7.9b} positivity properties of $[Q(a^*),Q(a)]$ will be crucial.  We shall completely abandon the use of the pseudodifferential calculus, in
particular we shall not use the symbols $X$ and $Y$. Rather we are
going to do exact calculations of the above commutator.
Note that $p_{kl}:=[a_k,a_l^*]=2\sigma\partial _l\omega_k$,
  and thus $P:=(p_{kl})=2\sigma\omega'\geq c\sigma r^{-1-\epsilon}>0$.  

 From  \eqref{eq:16pois}  (or for other reasons) one might
guess  that to  ``leading
order'' 
\begin{align}\label{eq:posFirst}
 [Q(a),Q(a^*)]\approx  2\sigma Q'(a)\omega'Q'(a^*)^T\geq 0.
\end{align}  However this analogy with
the previous sections turns out to be somewhat misleading, or at least
insufficient, for the
problem at hand. From the viewpoint
of the calculus of pseudodifferential operators there is a competition  in  a symbol between the
behaviour as the phase-space variables $\to \infty$ and
the behaviour when $\sigma\to \infty$ and it is  natural to use a
suitable parameter-dependent calculus, see Subsubsection \ref{subsubsection:Limitations of method}
 for a possible candidate. But even with such  a device this competition appears  too  subtle
to be resolved  (at least for us)  and consequently  we are going  to do
{\it exact}  calculations on  the commutator $[Q(a),Q(a^*)]$. Those  belong to the functional calculus rather than the pseudodifferential calculus. We shall derive
a combinatorial   formula in which the ``total amount'' of positivity in
$[Q(a),Q(a^*)]$ is explicitly exposed. The   expression in 
\eqref{eq:posFirst} turns out  to be only one out of in general many positive expressions
``hidden'' in the commutator. The other   expressions would  in the
framework of the previous sections be considered as harmless lower
order 
terms. In the present context they also  have
``lower order''. Nevertheless  as the reader will see we  will need all  of them.

\subsection{Calculation of a commutator}\label{subsec:Calculation of a commutator}

For each $m \ge 1$, let $J_m = (j_1,...,j_m)$  and $K_m =
(k_1,...,k_m)$ be $m$-tuples of numbers in $\{1,\dots, d\}$.  In this section we will prove the following formula:

\begin{align}\label{commutator formula}
&[Q(a),Q(a^*)] = F + E;  \\
&F=\sum_{m \ge1, J_m,
  K_m}(m!)^{-1}(\partial_{j_1}\cdots \partial_{j_m}Q)(a^*)\parbb {\prod_{l= 1}^m p_{j_lk_l}}(\partial_{k_1}\cdots \partial_{k_m}Q)(a), \notag \\
&E =  \sum_{m \ge1, J_m, K_m, \bar \alpha + \bar \beta \ne
  0}c^{\bar \alpha,\bar \beta}_{J_m,K_m}(\partial^{\alpha}\partial_{j_1}\cdots \partial_{j_m}Q)(a^*)P^{\bar
  \alpha,\bar \beta}_{J_m,K_m}(\partial^{\beta}\partial_{k_1}\cdots \partial_{k_m}Q)(a), \notag 
\end{align}
where the summation parameters $\bar \alpha = (\alpha_1,...,\alpha_m)$ and $\bar \beta
=(\beta_1,...,\beta_m)$ for the sum  $E$ denote  arbitrary $m$-tuples of multi-indices,
$\alpha=\Sigma_{l=1}^m \alpha_l$,   $\beta= \Sigma_{l=1}^m \beta_l$,
and 
$P^{\bar
  \alpha,\bar \beta}_{J_m,K_m}=
(\partial^{\alpha_1+\beta_1}p_{j_1k_1})\cdots( \partial^{\alpha_m+\beta_m}p_{j_mk_m})$.
\vspace{2mm}
We will not need an explicit expression for the combinatorial coefficient $c^{\bar
  \alpha,\bar \beta}_{J_m,K_m}$ because for $\sigma$ large, the term $E$ will be seen to be negligible.

To prove (\ref{commutator formula}) let $\adop_b(c) = [b,c], \rop_b(c) =
cb, \lop_b(c) = bc$.  Note that for commuting $b_1 \mand b_2$ all the
operators $\ad_{b_1}, \ad_{b_2}, \rop_{b_1}, \rop_{b_2}, \lop_{b_1},
\lop_{b_2}$ commute.  For $f$ a polynomial in $d$ variables and
$b_1,b_2,...,b_d$ commuting operators, we note the Taylor type formulas
\begin{subequations}
\begin{align}
  \begin{split}
\label{commutationformula}
[f(b),c] &= (f(\ad_b +\rop_b) - f(\rop_b)) c\\&= \Sigma_{\alpha \ne 0}(\alpha!)^{-1}
\ad_b^{\alpha}(c)\partial^{\alpha}f(b),
  \end{split} 
\end{align}
\begin{align}
\begin{split}
\label{commutationformulaB}
[f(b),c] &=  (f( \lop_b) - f(-\ad_b + \lop_b)) c\\ & = \Sigma_{\alpha \ne
  0}\tfrac{(-1)^{|\alpha| +
  1}}{\alpha!} \partial^{\alpha}f(b)\ad_b^{\alpha}(c).
\end{split}  
\end{align} 
\end{subequations}

We will also need the Leibniz type formula

\begin{align}\label{eq:Leib}
\ad_b^{\alpha}(cd) = \sum_{\gamma}\binom{\alpha}{\gamma}\ad_b^{\alpha - \gamma}(c) \ad_b^{\gamma}(d).
\end{align}

By \eqref{commutationformula} and \eqref{commutationformulaB} 
\begin{align}\label{eq:7}
& [Q(a),Q(a^*)] =  \Sigma_{\alpha \ne 0} (\alpha!)^{-1}\parb{\ad_a^{\alpha}Q(a^*)}Q^{(\alpha)}(a),\\
 &[a_j,Q(a^*)] = \Sigma_{\alpha \ne 0} \tfrac{(-1)^{\alpha}}{\alpha!}Q^{(\alpha)}(a^*) \ad_{a^*}^{\alpha}a_j.\label{eq:16}
\end{align} Here and below we denote $g^{(\alpha)} =\partial^{\alpha} g 
$, $(-1)^\alpha=(-1)^{|\alpha|}$ and $\ad_b^\alpha c=\ad_b^\alpha (c)$.

We will use  the summation rule
\begin{equation} \label{zeta}
\Sigma_{\alpha \ne 0}f(\alpha) = \Sigma_{\beta , k} \,\zeta(\beta + e_k) f(\beta + e_k),
\end{equation}
where for $\alpha \ne 0$,  $\zeta(\alpha)^{-1}$ = the number of $j$'s
with $\alpha_j >0$, and  $\{e_1,...,e_d\}$ is the standard basis for $\mathbb{R}^d$.
 Note that 
the number of pairs $(\beta, k)$ such that  $\alpha = \beta + e_k$ is the number of $j$'s such that $\alpha_j > 0$.

Introducing also the notation $d(\alpha)= (\alpha!)^{-1} \zeta(\alpha)$ it follows that 
\begin{align}\label{eq:18}
  \begin{split}
   &\ad_a^{\alpha+e_j}Q(a^*) \\& = \sum_{\beta,\gamma,\mu,k, \gamma+
     \mu = \alpha}\tfrac{\ad_a^{\mu}Q^{(\beta +e_k)}(a^*)}{\mu!}
   \bigg\{(-1)^{\beta} d(\beta+e_k)\tfrac{\alpha!}{\gamma!}\ad_a^{
     \beta+ \gamma }p_{jk}\bigg\}.
  \end{split}
\end{align} Here we used \eqref{eq:Leib}, \eqref{eq:16},
\eqref{zeta}, the computation $\ad_{a^*_k}a_j= -p_{jk}$ and whence
that
\begin{align*}
  \ad^{\gamma}_{a}\ad^{\beta+e_k}_{a^*}a_j=-\ad^{\gamma+\beta}_{a}p_{jk}=-\ad^{\beta+\gamma}_{a}p_{jk}.
\end{align*}

Thus for  $f(\mu)=Q^{(\mu)}(a)$ (or any other   operator-valued
function $f$)  and with
\begin{align*}
  \lambda(\beta,\gamma,\mu,j,k) := \tfrac{(\gamma + \mu)!}{\gamma!}d(\beta + e_k) d(\gamma + \mu + e_j),
\end{align*}

\begin{align}\label{eq:19}    
  \begin{split}
&\sum_{\mu \ne 0}\tfrac{\ad_a^{\mu}Q(a^*)}{\mu!}f(\mu) \\
&=\sum_{\beta,\gamma,j,k,\mu}\tfrac{\ad_a^{\mu}Q^{(\beta + e_k)}(a^*)}{\mu!}\bigg\{(-1)^{\beta} \lambda(\beta,\gamma,\mu,j,k)\ad_a^{\beta+\gamma}(p_{jk})f(\gamma + \mu + e_j)\bigg\} \\
&= \sum_{\beta,\gamma,j,k}Q^{(\beta + e_k)}(a^*)\bigg \{(-1)^{\beta}d(\beta+e_k)d(\gamma+e_j)\ad_a^{\beta+\gamma}(p_{jk})f(\gamma + e_j)\bigg\} \\
&+\sum_{\beta,\gamma,j,k}\sum_{\mu \ne 0}\tfrac{\ad_a^{\mu}Q^{(\beta +
    e_k)}(a^*)}{\mu!}\bigg \{(-1)^{\beta} \lambda(\beta,\gamma,\mu,j,k)\ad_a^{\beta+\gamma}(p_{jk})f(\gamma  + e_j+ \mu)\bigg\}.
  \end{split}
\end{align}
 Here we have used \eqref{zeta} and \eqref{eq:18}.

From \eqref{eq:19}  we can proceed inductively.
Introduce the notation
\begin{align*}
 B_l &= (\beta_1,,,\beta_l),\quad  \Gamma_l = (\gamma_1,,,\gamma_l),\\
 J_l &= (j_1,,,j_l), \quad  K_l = (k_1,,,k_l),\\
p_{jk\beta \gamma} &= \partial^{\beta + \gamma} p_{jk},\quad   D =
-\i \partial.
\end{align*} Here the components of $ B_l $ and $\Gamma_l $ are
multi-indices, while the components of $J_l$ and $K_l$ are numbers in $\{1,\dots, d\}$.

Repeatedly using \eqref{eq:19}  we obtain
\begin{align}   \label{general}
[Q(a),Q(a^*)]&\\ = \sum_{\beta,\gamma,j,k} &d(\beta+e_k) d(\gamma+e_j) (D^{(\beta + e_k)}Q)(a)^*p_{jk\beta \gamma}(D^{(\gamma + e_j)}Q)(a) \notag \\
 +\sum _{m\geq 2, B_m,\Gamma_m,J_m,K_m} &C^{B_m,\Gamma_m}_{J_m, K_m}(D^{\Sigma_{l = 1}^m(\beta_l + e_{k_l})}Q)(a)^*\parbb{\prod_{l= 1}^m p_{j_lk_l\beta_l \gamma_l}}(D^{\Sigma_{l = 1}^m(\gamma_l + e_{j_l})}Q)(a). \notag 
\end{align}
Here 
\begin{align*}
&C^{B_m,\Gamma_m}_{J_m,K_m} \\ & =\parbb{\prod_{l=1}^m\frac{ d(\beta_l + e_{k_l})}{\gamma_l!}} \prod_{l=1}^m\parbb{\parb{\gamma_l + \Sigma_{k=l+1}^m(\gamma_k + e_{j_k})}! d\parb{\Sigma_{k=l}^m(\gamma_k + e_{j_k})}},
\end{align*}
where the empty sum, $\Sigma_{k=m+1}^m$,  is by convention $=0$. 
Note that if $\beta_j = \gamma_j =0$ for all  $j$ then  we have 

\begin{align*}
C^{B_m,\Gamma_m}_{J_m, K_m}= C_{J_m}: = \parb{(\Sigma_{k = 1}^m
  e_{j_k})! }^{-1}\prod_{l=1}^m\zeta(\Sigma_{k=l}^m e_{j_k}).
\end{align*}

To compute the first term in (\ref{commutator formula}) note that in
(\ref{general}) we can replace $C_{J_m}$ by its average over
permutations.  To compute this we  use \eqref{zeta} and compute as a formal sum
\begin{align*}
\Sigma_{\alpha} f(\alpha) &=    f(0) + \Sigma_{\alpha, j} f(\alpha+ e_j) \zeta(\alpha + e_j)\\
& = f(0) +  \Sigma_j f(e_j) \zeta( e_j) + \sum_{\alpha, j_1, j_2} f(\alpha+ e_{j_1} +e_{j_2}) \zeta(\alpha + e_{j_1} + e_{j_2})\zeta(\alpha + e_{j_2}) \\
& = f(0)+ \sum_{m\geq 1,J_m} \parbb{f(\Sigma_{l=1}^me_{j_l})\prod_{k=1}^m\zeta(\Sigma_{l=k}^m e_{j_l})}.
\end{align*}

We set $f(\alpha) = x^{\alpha}/\alpha!$ and obtain
\begin{align*}
 e^{(x_1 + x_2 + ... x_d)} = &1 + \sum_{m\geq 1,J_m} \frac{x^{(e_{j_1} + e_{j_2} + ... +e_{j_m})}}{(e_{j_1} + e_{j_2} + ... +e_{j_m})!}\prod_{k=1}^m\zeta(\Sigma_{l=k}^m e_{j_l}) \\
 & = 1 +  \sum_{m\geq 1,J_m} C_{J_m} x^{(e_{j_1} + e_{j_2} + ... +e_{j_m})}.
\end{align*}
Differentiating and then setting $x=0$ gives
\begin{equation*}
1 = \frac{\partial^m e^{(x_1 + x_2 + ... x_d)} }{\partial x_{k_1}...\partial x_{k_m}}|_{x=0} = \sum_{\sigma \in \mathcal S_m} C_{k_{\sigma(1)}, ..., k_{\sigma(m)}}.
\end{equation*}

It follows that 
\begin{equation*} 
F = \sum_{m \ge1} (m!)^{-1}\sum_{J_m, K_m}(D^{\Sigma_{l = 1}^me_{k_l}}Q)(a)^*\parbb{\prod_{l= 1}^m p_{j_lk_l}}(D^{\Sigma_{l = 1}^m e_{j_l}}Q)(a), 
\end{equation*}
which  is (\ref{commutator formula}) with $E$ given with reference to (\ref{general}).

\subsection{Proof of Theorem \ref{thm:12.7.2.7.9b}}

The positivity of $F$ comes from the inequality 
\begin{align}\label{eq:22}
P \otimes \cdots \otimes P \ge c\sigma^mr^{-(1+\epsilon)m}I
\end{align}
on $\otimes_{j=1}^m l^2(\{1,...,d\})$ for some $c>0$.  This gives 
\begin{align}\label{simppositivity}
CF \ge \sum _{\alpha \ne 0}\sigma^{|\alpha|}\partial^{\alpha}Q(a^*)r^{-(1+\epsilon)|\alpha|}\partial^{\alpha}Q(a).
\end{align}

A typical term in $E$ can be written
\begin{align*}
T = \sigma^m \partial^{\alpha + \mu}Q(a^*)R_{\alpha,\beta,\mu,\nu} \partial^{\beta + \nu}Q(a),
\end{align*}
 where $|\alpha| = |\beta| = m$,  $\mu + \nu \ne 0$, and
 $|R_{\alpha,\beta,\mu,\nu}| \le C r^{-m - |\mu| - |\nu|}$.  It follows that 
 \begin{align*}
\pm \Re(T) \le &C\parb{\sigma^m \lambda \partial^{\alpha + \mu}Q(a^*)r^{-m - |\mu| - |\nu|}\partial^{\alpha + \mu}Q(a) + \\ \notag
& \sigma^m \lambda^{-1} \partial^{\beta + \nu}Q(a^*)r^{-m - |\mu| - |\nu|}\partial^{\beta + \nu}Q(a)}.
\end{align*}

If one of the multi-indices $\mu$ or $\nu$ is zero, say $\nu = 0$, then take $\lambda = \sigma^{1/2}$. Otherwise take  $\lambda = 1$.  If we take $\epsilon < 1/q $ clearly this term is negligible compared to $F$ for large $\sigma$.  It follows that $E$ is negligible compared to $F$ for large $\sigma$.

Let us now prove Theorem \ref{thm:12.7.2.7.9b}.  We have 
$(Q(a^*) + V_1 - \lambda)\phi_{\sigma} = - V_2\phi_{\sigma}$ which gives
\begin{align}\label{eq:20}
  \begin{split}
  (\phi_{\sigma},[Q(a),Q(a^*)]\phi_{\sigma}) &+ 2\Re(\phi_{\sigma},[Q(a), V_1]\phi_{\sigma}) + ||(Q(a) + V_1 - \lambda)\phi_{\sigma}||^2\\&= ||V_2\phi_{\sigma}||^2.  
  \end{split}
  \end{align}

We use (\ref{commutationformula}) to compute $[Q(a), V_1]$:
\begin{align}
[Q(a), V_1] &= \Sigma_{\alpha \ne 0} (\alpha!)^{-1}(\ad_a^{\alpha}V_1) \partial^{\alpha}Q(a)\notag \\
& =  \Sigma_{\alpha \ne 0} (\alpha!)^{-1}((-\i\partial)^{\alpha}V_1) \partial^{\alpha}Q(a).\label{QV_1formula}
\end{align}

With \cs we can bound 
\begin{align*}
-2\Re(\phi_{\sigma},[Q(a), V_1]\phi_{\sigma}) \le C \Sigma_{1 \le |\alpha| \le q}\parb{r^{(1+\epsilon)|\alpha|}(\partial^{\alpha}V_1)^2 + \partial^{\alpha}Q(a^*)r^{-(1+\epsilon)|\alpha|}\partial^{\alpha}Q(a)}.
\end{align*}
The fact that $\phi_{\sigma} = 0$ follows by taking $\epsilon$ small and then using the formula (\ref{simppositivity}).   In addition the terms with $m = q$ in  (\ref{simppositivity}) are used to bound  $|V_2|^2$ and the terms $r^{(1+\epsilon)|\alpha|}(\partial^{\alpha}V_1)^2$.
This completes the proof of Theorem \ref{thm:12.7.2.7.9b}.

In the next subsection we display additional positivity of $[Q(a^*),Q(a)]$ by giving a symmetrized estimate.  This will be important in proving Theorem \ref{better}.
\subsection{Symmetrized estimate}

Abbreviating $r^{-(1+\epsilon)|\alpha|} = R_{\alpha}, \partial^{\alpha}Q = Q^{(\alpha)}$ we will show here that for some $C >0$ and all large $\sigma$ \begin{align}\label{sympositivity}
C[Q(a), Q(a^*)] \ge  \Sigma_{\alpha \ne 0}\,\sigma^{|\alpha|}\parb{Q^{(\alpha)}(a^*)R_{\alpha}Q^{(\alpha)}(a)+
Q^{(\alpha)}(a)R_{\alpha}Q^{(\alpha)}(a^*)}.
\end{align}

For any $m\geq 1 $ we abbreviate $J = (j_1,...,j_m),
K=(k_i,...k_m),\partial_J = \partial_{j_1} \cdots\partial_{j_m},
\newline P_{JK} = p_{j_1k_1}\cdots p_{j_mk_m}$.  We introduce
\begin{align*}
  F_\lef&=\sum_{m,J,K} (m!)^{-1}\partial_J Q(a^*)P_{JK}\partial_K
  Q(a),\\
F_\righ&=\sum_{m,J,K} (m!)^{-1} \partial_KQ(a)P_{KJ}\partial_J Q(a^*).
\end{align*} Clearly $F=F_{\lef}$, and 
by \eqref{eq:22}  and \eqref{simppositivity}
\begin{subequations}
\begin{align}\label{eq:24}
  F_\lef&\geq c\Sigma_{\alpha \ne
    0}\, \sigma^{|\alpha|}Q^{(\alpha)}(a^*)R_{\alpha}Q^{(\alpha)}(a),
\\F_\righ&\geq c\Sigma_{\alpha \ne 0}\,\sigma^{|\alpha|}Q^{(\alpha)}(a)R_{\alpha}Q^{(\alpha)}(a^*).\label{eq:25}
\end{align}  
\end{subequations}
Now for any term in $F_\lef$ we decompose using the symmetry $P_{JK}=
P_{KJ}$
\begin{align*}
  &\partial_J Q(a^*)P_{JK}\partial_K Q(a)-[\partial_J Q(a^*),P_{JK}]\partial_K Q(a)+P_{JK}[\partial_K Q(a),\partial_J Q(a^*)]\\ 
&  =\partial_K Q(a)P_{KJ}\partial_J Q(a^*)   - [\partial_K Q(a),P_{JK}] \partial_J Q(a^*),
\end{align*} and write this formula as 
\begin{align*}
   T^\lef_{m,J,K}=T^\righ_{m,J,K}.  
  \end{align*}
It suffices to show that for all large $\sigma$ and all small
$\epsilon$
\begin{subequations}
\begin{align}
  \label{eq:21}
  \sum_{m,J,K} (m!)^{-1}\Re\parb{T^\lef_{m,J,K}}&\leq C F_\lef,\\
\sum_{m,J,K} (m!)^{-1}\Re \parb{T^\righ_{m,J,K}}&\geq \tfrac 12 F_\righ.\label{eq:23}
\end{align} 
\end{subequations} The proof of \eqref{eq:21} and \eqref{eq:23} is
given by using \eqref{eq:24} and \eqref{eq:25}, respectively. Let us
here derive \eqref{eq:21} only.

For the middle term we 
calculate using  \eqref{commutationformulaB}
\begin{align}
 -[\partial_J Q(a^*),P_{JK}] =  \Sigma_{\alpha \ne 0}\tfrac
 {(-1)^{\alpha}} {\alpha!}\partial^{\alpha}\partial_JQ(a^*) \ad_{a^*}^{\alpha}P_{JK},
\end{align}
which gives

\begin{align}\label{error_1}
  \begin{split}
-\sum_{m,J,K} (m!)^{-1}&\Re\parb{[\partial_J Q(a^*),P_{JK}]\partial_K Q(a)}\\ \le C\sigma^{m+1/2}&\sum_{m,J,K,\alpha \ne 0}\partial^{\alpha}\partial_JQ(a^*)r^{-(m +3|\alpha|/2)}\partial^{\alpha}\partial_J Q(a) 
\\
& +C\sigma^{m-1/2}\sum_{m,J,K,\alpha \ne 0}\partial_KQ(a^*)r^{-(m
  +|\alpha|/2)}\partial_KQ(a)\\
&\leq \tfrac 14  F_\lef+\tfrac 14  F_\lef.  
\end{split}
\end{align}
In the last step we used \eqref{eq:24} and needed large $\sigma$
large and 
$\epsilon$  small.
 It remains to consider the last term.  Using (\ref{commutator formula}) we have 
\begin{align*}
 &\sum_{m,J,K} (m!)^{-1}\Re\parb{P_{JK}[\partial_J Q(a^*),\partial_K
   Q(a)]}\\ &=  \sum_{m,J,K} (m!)^{-1}
\sum_{ l\ge 1,J'_l,K'_l}(l!)^{-1}\Re\parb{P_{JK}\partial_{J'_l} \partial_J Q(a^*)P_{J'_l,K'_l}\partial_{J'_l}\partial_JQ(a)} + E',
\end{align*}
where $E'$ comes from $E$ in (\ref{commutator formula}).  Once
$P_{JK}$ is commuted past $\partial_{J'_l} \partial_J Q(a^*)$ we get
for the fixed $m$ and $l$ portion of the summation in  the  first term on the right side
\begin{align*}
  C_{m,l}\Sigma_{J,K,J'_l,K'_l }\partial_{J'_l} \partial_J Q(a^*)P_{J,K}P_{J'_l,K'_l}\partial_{J'_l}\partial_JQ(a) =  C_{m,l}\Sigma_{L,M} \partial_{L} Q(a^*)P_{L,M}\partial_MQ(a),
\end{align*}
 where $L = (j_1,...,j_{l+m}) \mand M= (k_1,...k_{l+m})$ and the
 $j_i's \mand k_i's$ are summed over. The contribution from the resulting expression  can be
 estimated  by a  multiple of $F_\lef$.  The contribution from the
 commutator and the term $E'$ are  handled as
 in (\ref{error_1}).   Thus combining with \eqref{error_1} we obtain
 \eqref{eq:21}.

\subsection{Proof of Theorem \ref{better}} \label{subsection:Improvement for}

Let $f_m = r^{-(1+\epsilon)m/2}$.  We first note that using the same
ideas (commutation and \cS) as in the last two subsections we can equivalently write (for large $\sigma$, small enough $\epsilon$, and some $C>0$)

\begin{align} \label{Sdef}
&C[Q(a),Q(a^*)] \ge S;\\
&S = \sum_{m=1} \sigma^m\sum_{l_1,\dots,l_m\leq
  d}\parbb{|\parb{\partial_{{l_1}}\cdots \partial_{{l_m}}Q(a^*)}f_m(r)|^2+|\parb{\partial_{l_1}\cdots \partial_{l_m}Q(a)}f_m(r)|^2}. \notag 
\end{align}

We need a more efficient extraction of positivity from (\ref{Sdef}) than is immediately evident from this lower bound.  We will also need a formula for $[Q(a), V_1]$ different from (\ref{QV_1formula}) which was used in the proof of Theorem \ref{thm:12.7.2.7.9b}.

Thus for $Q(\xi) = \xi^2$, we have $\partial_jQ(\xi) = 2\xi_j$ so that 
\begin{align*}
  S &= \Sigma_j4\sigma f_1(a^*_ja_j + a_ja^*_j)f_1 + 8d\sigma^2 f_2^2\\ &= 8\sigma f_1(p^2 +\sigma^2 \omega^2)f_1 + 8d\sigma^2 f_2^2.
\end{align*}
 Thus for large $\sigma$ and some $C>0$
\begin{align}\label{square}
CS \ge \sigma^2f_1^2. 
 \end{align}
 Moving on to the commutator of $Q(a)$ with $V_1$ we have 
\begin{align}\label{eq:8}
\Re[Q(a),V_1] =\Re \Sigma_j (a_j[a_j, V_1] + [a_j,V_1]a_j) = -2\sigma \omega \cdot \nabla V_1.
\end{align}
Again applying \eqref{eq:20}, the result for $Q(\xi) =
\xi^2$ follows from (\ref{square}) and \eqref{eq:8}.  

We now consider $Q(\xi) = (\xi^2)^2$. 
Note that $\partial_j Q(\xi) =  4\xi^2\xi_j$ and $\partial_j^2 Q(\xi)
= 8(\xi^2_j+\xi^2 /2)$.  Whence for any operator $P\geq 0$
\begin{align*}
  64^{-1}\Sigma_j (\partial_j^2Q(a)P\partial_j^2Q(a^*)
+ \partial_j^2Q(a^*)P\partial_j^2Q(a))\geq \Sigma_j \parb{ a_j^2P (a^*_j)^2+(a^*_j)^2Pa_j^2}.
\end{align*}
We will also use  the following identity for an operator $b$
\begin{align*}
&b^2(b^*)^2 + (b^*)^2b^2=\\
&\parbb {b(bb^* +b^*b)b^* + b^*(bb^* + b^*b)b + \ad_b^2(b^*)b^* +
  \ad_{b^*}^2(b)b +[b,b^*]^2}/2.
\end{align*}

Applied to $P=I$ and $b=a_j$ it follows that 
\begin{align*}
&64^{-1}\Sigma_j (\partial_j^2Q(a)\partial_j^2Q(a^*)
+ \partial_j^2Q(a^*)\partial_j^2Q(a)) \\& \geq  
\Sigma_j\parb {a_j(p_j^2 +\sigma^2\omega_j^2)a_j^* + a_j^*(p_j^2
  +\sigma^2\omega_j^2)a_j +\i\sigma\partial_j^2\omega_j(a_j - a_j^*) +
  2\sigma^2(\partial_j\omega_j)^2} \notag\\
&\ge 2\Sigma_j\parb { \sigma^4\omega_j^4 +  \sigma^2 \partial_j(\omega_j\partial_j\omega_j)}.
\end{align*}

If we add a suitable multiple of the $m=4$ term of (\ref{Sdef}) we obtain
\begin{align}\label{4thpower}
CS \ge \sigma^4f_2^2.
\end{align}

We now go on to compute $[Q(a),V_1]$.
We have 
\begin{align*}
&[(a^2)^2,V_1] \\&= 2\Sigma_j\parb {[a_j,V_1]a_ja^2+a^2a_j[a_j,V_1]}+\Sigma_{i,j}\parb {(\ad^2_{a_j}V_1) a^2_i -a^2_i (\ad^2_{a_j}V_1)} \\
&=-\tfrac \i2\Sigma_j  \parb {\partial_jV_1\partial_jQ(a)
  + \partial_jQ(a) \partial_jV_1}+ \Sigma_{i}\parb {a^2_i(\Delta V_1) - (\Delta V_1)a^2_i}.
\end{align*}
We bound
\begin{align*}
&-\Im \Sigma_j  \parb {\partial_jV_1\partial_jQ(a) + \partial_jQ(a) \partial_jV_1}   \notag \\
&\le \Sigma_j  \parb {\partial_jQ(a^*)f_1^2\partial_jQ(a)  + \partial_jQ(a)f_1^2\partial_jQ(a^*)+(f_1^{-1}\partial_jV_1)^2} \\&\le C\sigma^{-1}S, 
\end{align*}
where we have taken $\epsilon$ small, $\sigma$ large and used
(\ref{4thpower}). Similarly
\begin{align*}
&-2\Re \Sigma_{i} \parb {a_i^2(\Delta V_1)
  + (\Delta V_1)a_i^2}\\& \le \Sigma_i \parb{(a_i^*)^2 f_2^2a_i^2 +
  a_i^2 f_2^2(a_i^*)^2} + 2d(f_2^{-1}\Delta V_1)^2\\& \leq 64^{-1}\Sigma_i\parb {\partial_i^2Q(a^*)f_2^2\partial_i^2Q(a) +\partial_i^2Q(a)f_2^2\partial_i^2Q(a^*) } + 2d(f_2^{-1}\Delta V_1)^2\\&\le C \sigma^{-2} S.
\end{align*}

Putting these estimates together gives Theorem \ref{better} for $Q(\xi) = (\xi^2)^2$.

\subsubsection{Limits  of the method, 
  examples} \label{subsubsection:Limitations of method}
We continue the discussion of the examples treated above. Introduce
$\inp{\xi}_\sigma=(\xi^2+\sigma^2\omega^2)^{1/2}$ and
$\inp{p}_\sigma=(p^2+\sigma^2\omega^2)^{1/2}$. 
For $Q(\xi) = \xi^2$ we found
 the lower bound
\begin{align}
  \label{eq:26}
  C S\geq \sigma f_1\inp{p}_\sigma^2f_1+ \sigma^2 f^2_2.
\end{align} For $Q(\xi) = (\xi^2)^2$ we have  the lower bound
\begin{subequations}
\begin{align}
  \label{eq:27}
  C S\geq \sigma^2 f_2\inp{p}_\sigma^4f_2+ \sigma^4 f^2_4,
\end{align} which is an extension of \eqref{4thpower} and follows
from its proof.

Letting 
\begin{align*}
  g=r^{-2}\d x^2+\parb{\xi^2+\sigma^2}^{-1}\d \xi^2;\;\sigma>1,
\end{align*} the symbol of $S$ for $Q(\xi) = \xi^2$ is  in the uniform
parameter-dependent class (cf. \cite[Chapt. XVIII]{Ho})
\begin{align*}
  S_\unif(\sigma r^{-1}\parb{\xi^2+\sigma^2},g),
\end{align*} and  for $Q(\xi) = (\xi^2)^2$ in 
\begin{align*}
  S_\unif(\sigma r^{-1}\parb{\xi^2+\sigma^2}^3,g).
\end{align*} Comparing with \eqref{eq:26} and \eqref{eq:27} we see
that essentially we  got an {\it elliptic} estimate in the case of
$Q(\xi) = \xi^2$ (there is  a loss of the small power $r^\epsilon$ and a slight
modification  at the critical point $x=0$),  while  we only  got a {\it subelliptic} estimate in
the case of $Q(\xi) = (\xi^2)^2$. In the latter case possibly ``ellipticity'' would
be  the stronger bound
\begin{align}
  \label{eq:27b}
  C S\geq \sigma f_1\inp{p}_\sigma^6f_1+ \sigma^4 f^2_4.
\end{align} Somehow we lost a factor of
$r^{1+\epsilon}\sigma^{-1}\inp{p}^2_\sigma\approx r\sigma^{-1}\inp{p}^2_\sigma$, and it is
natural to ask if \eqref{eq:27} can be improved perhaps up to the
 optimal type  bound \eqref{eq:27b}?  We will show  this is not
 possible, in   particular we will show that our bound \eqref{eq:27} can be
 considered ``optimal''. Note that  the bound \eqref{eq:27b}
 would lead to 
\begin{align}
  \label{eq:27bc}
  C S\geq \sigma^{7} |\omega|^6r^{-1-\epsilon},
\end{align} while  \eqref{eq:27} implies
\begin{align}
  \label{eq:27bcc}
  C S\geq \sigma^{6} |\omega|^4r^{-2-2\epsilon}.
\end{align}
\end{subequations}
\begin{lemma}\label{lemma:limit-meth-exampl} Consider $Q(\xi) =
  (\xi^2)^2$.  Both of the following assertions 
 are false. 

  \begin{subequations}
For some $s\in \R$ and  $t\geq 0$ there exists $\epsilon_0\in
  (0,1)$ such that for all $\epsilon\in (0,\epsilon_0]$ there are
  constants $C_\epsilon,\sigma_\epsilon>1$:
  \begin{align}
    \label{eq:31a}
    C_\epsilon[Q(a),Q(a^*)]\geq \sigma^{1+s} |\omega|^tr^{-2+\epsilon}\mforall
    \sigma\geq \sigma_\epsilon.
  \end{align} 

For some $s>5$  and $t\geq 0$ there exists $\epsilon_0\in
  (0,1)$ such that for all $\epsilon\in (0,\epsilon_0]$ there are
  constants $C_\epsilon,\sigma_\epsilon>1$:
  \begin{align}
    \label{eq:31b}
    C_\epsilon[Q(a),Q(a^*)]\geq \sigma^{1+s} |\omega|^tr^{-2}\mforall
    \sigma\geq \sigma_\epsilon.
  \end{align}  
  \end{subequations}
\end{lemma}
\begin{proof}
  We  introduce  a state of the form
\begin{align*}
  \psi_\sigma(x)=k^{-(d-1)/2}Y_l(\hat x)\phi((|x|-k)/m),
\end{align*} where  $Y_l$ is a spherical harmonic and the
indices $k,l,m>0$ are large. More precisely we  take
$k= \sigma^{(5+\epsilon+|s|)/\epsilon}$,  
$m=\sqrt{k/\sigma}$ and  $l$ to be    the integer part of
$\tilde l$ which is the unique positive solution to the equation
\begin{align*}
  \tfrac {\tilde l(\tilde l+d-2)}{k^2}=\sigma^2\omega(ke)^2,
\end{align*} where $e$ is an arbitrary unit vector in $\R^d$. 
 Fix $\phi\in C^\infty_\c(\R_+)$ normalized, 
$\|\phi\|_{L^2}=1$. Note that $\psi_\sigma$ defined this way is
approximately normalized. 

Corresponding to \eqref{eq:31a} and \eqref{eq:31b}
\begin{subequations}
  \begin{align}
    \label{eq:31aB}
    \inp{\sigma^{1+s} |\omega|^tr^{-2+\epsilon}}_{\psi_\sigma}\approx
    \sigma^{1+s}k^{-2+\epsilon},
  \end{align} 
  \begin{align}
    \label{eq:31bB}
    \inp{\sigma^{1+s} |\omega|^tr^{-2}}_{\psi_\sigma}\approx
    \sigma^{1+s}k^{-2}.
  \end{align}  
  \end{subequations}

 To calculate the expectation of the left hand side of
\eqref{eq:31a} (or \eqref{eq:31b}) we use
\eqref{commutator formula}. The leading term of the commutator is
\begin{align*}
  32\sigma (a^*)^2\Sigma_{i,j}a_i^*(\partial_i\omega_j)a_j a^2,
\end{align*} which using the notation  $p_\omega=1/2(\omega\cdot
p+p\cdot \omega)$ and the familiar formulas
\begin{align*}
  p^2 \parb{f(|x|)\otimes Y_l(\hat x)}&=\parbb{-f''(|x|)-\tfrac
    {d-1}{|x|}f'(|x|)+\tfrac {l(l+d-2)}{|x|^2}f(|x|)}\otimes Y_l(\hat
  x),\\
\i[p^2,p_\omega]&=2\Sigma_{i,j}p_i(\partial_i\omega_j)p_j-\tfrac 12 (\Delta^2r),
\end{align*} leads to the upper bound
\begin{align*}
   &\inp{[Q(a),Q(a^*)]}_{\psi_\sigma}\\& \leq C\sigma\parb
   {\|\inp{p}_\sigma r^{-1/2}(p^2-\sigma^2
  \omega^2)\psi_\sigma\|^2+\sigma^2\|\inp{p}_\sigma r^{-1/2}p_\omega
  \psi_\sigma\|^2}+C\|\sigma r^{-1}\inp{p}_\sigma^2\psi_\sigma\|^2\\
&\leq C\sigma^3k^{-1}\parb {\|\parb{\tfrac
   {l(l+d-2)}{|x|^2}-\sigma^2\omega^2}\psi_\sigma\|^2+\sigma^2m^{-2}} +C\sigma^6k^{-2}
\\&\leq C\sigma^3k^{-1}\parb{\sigma^4m^2/k^2+\sigma^2m^{-2}+\sigma^3k^{-1}}.
\end{align*}

In combination with \eqref{eq:31a}--\eqref{eq:31bB}  we thus obtain the
impossible bounds
\begin{align*}
  3C\sigma^6k^{-2}=C\sigma^3k^{-1}\parb{\sigma^4m^2/k^2+\sigma^2m^{-2}+\sigma^3k^{-1}}\geq 
  \begin{cases}
    \sigma^{1+s}k^{-2+\epsilon},\\
\sigma^{1+s}k^{-2}
  \end{cases}.
\end{align*} 
\end{proof}

\appendix

\section{The Weyl symbol of $Q(p+\i \nabla f(x))$}\label{sec:appA}

We give a combinatorial formula for the Weyl symbol of
\begin{align*}
  \Opw (b): = e^{f(x)}\Opw (a)e^{-f(x)},
\end{align*}
namely formally
\begin{align}
b(x,\xi) = a(x, \xi -\i\nabla_{y})\exp{(f(x-y/2) - f(x+y/2))}|_{y=0}.\end{align} 
In the special case that $a(x,\xi)$ is a polynomial, $Q(\xi)$, and $f \in C^{\infty}(\mathbb R^n)$ 
we have 
\begin{align}
  b(x,\xi) &=e^{-\i\nabla_{\xi}\cdot\nabla_y} e^{f(x-y/2) - f(x+y/2)}|_{y=0}Q(\xi)\nonumber\\
&= e^{f(x + \i \nabla_{\xi}/2) - f(x - \i \nabla_{\xi}/2)}Q(\xi)\nonumber\\
&= Q(\xi)+\sum_{k, n_1,n_3, \cdots, n_{2k+1};\, n_{2k+1}\geq 1}2^{n_1 + \cdots +
  n_{2k+1}}\\ 
&\frac{\parb{\frac{1}{1!}\frac{\i\nabla_{\xi} \cdot
      \nabla_x}{2}f(x)}^{n_1}}{n_1!}\frac{\parb {\frac{1}{3!}\parb{\frac{\i\nabla_{\xi} \cdot  \nabla_x}{2}}^3f(x)}^{n_3}}{n_3!}
 \cdots\frac{\parb {\frac{1}{(2k+1)!}\parb{\frac{\i\nabla_{\xi} \cdot
       \nabla_x}{2}}^{2k+1}f(x)}^{n_{2k+1}}}{n_{2k+1}!}Q(\xi).\nonumber
\end{align}

\end{document}